\documentclass[11pt]{amsart}
\usepackage{times}
\usepackage{mathrsfs}
\usepackage{algorithm}
\usepackage{algpseudocode}

\topmargin = - 40 pt 
\makeatletter
\providecommand{\@dottedtocline}[5]{%
  \@tocline{0}{1pt}{0pt}{}{\tocchapter}{#4}{#5}}
\makeatother

\usepackage{verbatim}
\usepackage[T1]{fontenc}
\usepackage[margin=28mm,heightrounded]{geometry}
\usepackage{amsmath,amssymb,amsthm,mathtools,booktabs,hyperref}
\usepackage{listings}
\hypersetup{colorlinks=true,linkcolor=blue,citecolor=blue,urlcolor=blue}
\newcommand{\Kis}{\mathsf {Kis}}
\newcommand{\Mat}{M}
\newcommand{\MM}{{\mathcal \Mat}}
\newcommand{\Bin}{H}

\newcommand{\ceven}{c_{\mathrm{even}}}
\newcommand{\codd}{c_{\mathrm{odd}}}
\newcommand{\PhiP}{\Phi_{+}}
\newcommand{\PhiM}{\Phi_{-}}
\newcommand{\esc}{g}
\newcommand{\lgTwo}{\log_2}

\newcommand{\cv}{\mathbf{c}}
\newcommand{\unit}{e_1}
\newcommand{\height}{\operatorname{ht}}
\newcommand{\supp}{\operatorname{supp}}
\newcommand{\norm}[1]{\left\lVert#1\right\rVert}
\newtheorem{theorem}{Theorem}[section]
\newtheorem{proposition}[theorem]{Proposition}
\newtheorem{lemma}[theorem]{Lemma}

\theoremstyle{definition}\newtheorem{definition}[theorem]{Definition}
\theoremstyle{remark}
\title[Cardinalities of Kiselman monoids and their parity asymptotics]{Cardinalities of Kiselman monoids\\ and their parity asymptotics}
\author{Alessandro D'Andrea}
\address{Alma Mater studiorum -- Universit\`a di Bologna\\
Dipartimento di Matematica\\
Piazza di porta San Donato 5, 40126 Bologna, Italy}
\email{a.dandrea@unibo.it}

\author{Salvatore Stella}
\address{Universit\`a degli Studi dell'Aquila\\
Dipartimento di Ingegneria e Scienze dell'Informazione e Mate\-ma\-tica\\
Via Vetoio, 67100 L'Aquila, Italy}
\email{salvatore.stella@univaq.it}
\thanks{Salvatore Stella was partially supported by PRIN 20223FEA2E and INdAM/GNSAGA}

\date{\today}
\begin{document}
\begin{abstract}
The cardinalities of Kiselman's semigroups $K_n$ are known to grow double exponentially in~$n$: the sequence $2^{-n/2}\log_2|K_n|$ converges~\cite{DS} along the even and odd $n$ subsequences to positive limits $\ceven$ and $\codd$ that appear to coincide.
In this paper we prove that they are indeed distinct.

Since Kiselman's semigroups are $\mathcal L$-trivial, the suffix relation defines a partial order on $K_n$.
We introduce the \emph{Kiselman transform} $\Kis(P)$ of a poset $P$, a construction recovering the poset associated with~$K_{n+1}$ from that associated with $K_n$.
Counting strictly increasing chains in $P$ and $\Kis(P)$ yields a linear recurrence governed by a symmetric infinite matrix with explicit binomial entries.
This can be used towards efficiently computing $|K_n|$, allowing the values listed in OEIS sequence A125625 to be extended beyond $n=14$.
A generating functions analysis gives formulas for the above asymptotic limits, together with rigorous numerical bounds establishing that
\[
\ceven-\codd
=1.9945066605\ldots\times10^{-6}.
\]
\end{abstract}
\maketitle

\section*{Introduction}
The Kiselman monoids generalize a semigroup introduced by Kiselman in the context of convex analysis~\cite{Ki}.
His construction concerns three idempotent operators on extended-real-valued functions: taking the greatest convex minorant, taking the greatest lower semicontinuous minorant, and replacing a function by the constant $-\infty$ whenever it assumes the value $-\infty$.
Together with the identity, these operators generate a monoid of eighteen elements.
Its defining relations generalize naturally to the family of monoids
\[
K_n=
\left\langle
 a_1,\ldots,a_n
\;\middle|\;
\begin{aligned}
a_i^2&=a_i &&(1\le i\le n),\\
a_i a_j a_i&=a_j a_i a_j=a_j a_i
&& (1\le i<j\le n)
\end{aligned}
\right\rangle,
\]
with Kiselman's original example corresponding to $K_3$.
These monoids, and their generalizations, have since been studied in connection with quotients of $0$-Hecke monoids~\cite{GM}, graph dynamical systems~\cite{CD}, and combinatorial normal forms~\cite{AD}.
Further developments include~\cite{An26,AVZ15,L25,MO19a}.

Kudryavtseva and Mazorchuk~\cite{KM} undertook a systematic study of the algebro-combinatorial structure of $K_n$.
In particular, they proved that each element admits a unique expression (that they call \emph{canonical}) of minimal length in the standard generators~\cite[Theorem~6]{KM}, and that all of Green's relations on~$K_n$ are trivial~\cite[Theorem~22]{KM}.
Consequently, $K_n$ is $\mathcal L$-trivial, and the relation
\[
x\le y
\quad\Longleftrightarrow\quad
K_n y\subset K_nx
\]
defines a partial order on $K_n$.
We call this the \emph{suffix order}: $x\le y$ precisely when $y$ can be obtained from~$x$ by multiplication on the left.
This partial order can also be described in terms of the \emph{update system} introduced in~\cite{CD}.

In Section~\ref{Kis_transform} we introduce a construction on posets abstracting the passage from the suffix poset of~$K_n$ to that of $K_{n+1}$.
Recall that $K_n$ identifies with the submonoid of $K_{n+1}$ generated by $a_1, \dots, a_n$.
The canonical-word theorem~\cite[Theorem 6]{KM} implies that each element in $K_{n+1} \setminus K_n$ has a unique canonical expression $ua_{n+1}v$, where $u,v\in K_n$ are canonical words in $\{a_1, \dots, a_{n}\}$.
The \emph{deletion properties} established in~\cite[Section 10]{KM}, after reversing the ordering of the generators, show that such an element is uniquely determined by the pair $(v,uv)$.
Consequently, the mapping
\[
K_{n+1}\setminus K_n \ni ua_{n+1}v\longmapsto(v,uv) \in K_n^2
\]
induces a bijection of sets $K_{n+1}\cong K_n\sqcup\{(x,y)\in K_n^2:x\le y\}$.
Indeed, every comparable pair occurs: if $y=ux$, the element $ua_{n+1}x$ corresponds to $(x,y)$.
One can therefore use $(x, ux)$ as an alternative notation for the element $ua_{n+1}x$.

The suffix order on $K_{n+1}$ can then be readily described in terms of the suffix order on $K_n$.
The subset~$K_n\subset K_{n+1}$ retains its partial order, no element from $K_{n+1} \setminus K_n$ sits below an element of $K_n$, and
\[
\begin{aligned}
p\le(x,y)
&\quad\Longleftrightarrow\quad p\le x,\\
(x,y)\le(u,v)
&\quad\Longleftrightarrow\quad
\bigl(x=u\ \text{and}\ y\le v\bigr)
\ \text{or}\ y\le u.
\end{aligned}
\]

Since this description involves only the poset structure on $K_n$, it can be generalized to arbitrary posets.
For a poset $P$, define its \emph{Kiselman transform} to be
\[
\Kis(P)=P\sqcup\{(x,y)\in P^2:x\le y\}.
\]
The proof that the above rules define a partial order on $\Kis(P)$ does not depend on the special Kiselman case.

The above construction also provides a recursive description of strictly increasing chains.
When $P$ is finite, the numbers of chains of each length in $\Kis(P)$ are obtained from those in $P$ by a fixed linear transformation.
More precisely, if
\[
  \cv_j(P)=\#\{x_1<\cdots<x_j\text{ in }P\}
\]
denotes the number of strictly increasing chains in $P$ with $j$ elements, and we write
\[
  \cv(P)=(\cv_1(P),\cv_2(P),\ldots)^{\mathsf T},
\]
then $\cv(P)$ has finite support, and its first entry is $|P|$.
In Section~\ref{chain_vectors} we prove that
\[
  \cv(\Kis(P))=\Mat\cv(P),
\]
where $\Mat = (\Mat_{ij})_{i,j\geq1}$ is an infinite \textbf{symmetric} matrix whose coefficients are simply expressed in terms of binomial coefficients
\[
  \Mat_{i,j}=\binom{i+j}{2i-j}.
\]
Iterating this transformation from the one-point poset provides an explicit method for computing the cardinalities of Kiselman monoids.
We use this in Section~\ref{explicit_bound} to show that one may obtain arbitrarily precise estimates of the asymptotic constants
\begin{equation}
\begin{aligned}
   \ceven
   &:=\lim_{n\to\infty}2^{-n}\lgTwo|K_{2n}|,\\
   \codd
   &:=\lim_{n\to\infty}2^{-n-\frac12}\lgTwo|K_{2n+1}|,
\end{aligned}
\end{equation}
introduced in~\cite{DS}.
The computations in this note, carried out with the assistance of ChatGPT, unexpectedly establish that $\ceven>\codd$.
The relative difference is remarkably small:
\[
\frac{\ceven}{\codd}-1\approx7.96\times10^{-7}.
\]

A further feature of the Kiselman transform is that posets of the same height\footnote{The height $\height(P)$ of a poset $P$ counts the elements in a longest chain.} have iterated cardinalities whose ratio stays between two fixed positive constants.
Their normalized logarithms therefore have the same asymptotic behaviour.
More precisely, if $P, Q$ are posets with the same height, then
\[
\lim_{n \to \infty} 2^{-n/2}(\lgTwo|\Kis^n(P)|-\lgTwo|\Kis^n(Q)|) = 0.
\]

As a final curious application of our techniques, we also compute $|K_n|$ modulo $4$.
Notice that the parity of $|K_n|$ had been previously determined in~\cite{An26Zero}.

In the Appendix, we describe an algorithm for the recursive computation of $|K_n|$, which may be used to extend the range tabulated in OEIS sequence A125625 beyond $n=14$.
The files attached to this arXiv submission contain Python and C implementations of it and its output up to $n=30$.

The historical notes attached to~\cite{OEIS} attribute the initial submission to Seidon Alsaody~\cite{Alsaody}, the value $|K_6|$ to a Magma computation by Klaus Brockhaus, and later extensions to M.~Selin.
The OEIS entry provides the C++ source at the (dead) link
\[
  \texttt{https://www.f.kth.se/\string~mxrten/kiselman\_semigroup.cpp}
\]
The original Selin source was thus not recovered.
The exact matrix computation in this note independently reproduces known values up to $n=14$ and extends the computations: it has been uploaded on the OEIS website.

\section{The suffix order on $K_n$ and the Kiselman transform of a poset}\label{Kis_transform}

\subsection{Canonical words in $K_n$}
Most of what is known about Kiselman's semigroups $K_n$ originates from Kudryavtseva and Mazorchuk's work~\cite{KM}: we summarize here the information that we are going to need in our construction.

Kiselman's semigroup \(K_n\) is the monoid defined by the presentation
\begin{equation}\label{eq:kiselman-presentation}
K_n=
\left\langle
 a_1,\ldots,a_n
\;\middle|\;
\begin{aligned}
a_i^2&=a_i &&(1\le i\le n),\\
a_i a_j a_i&=a_j a_i a_j=a_j a_i
&& (1\le i<j\le n)
\end{aligned}
\right\rangle.
\end{equation}
A word in the standard generators of $K_n$ is called \emph{canonical} if, between any two occurrences of the same generator $a_i$, there occur both a generator $a_j$ with $j<i$ and a generator $a_k$ with $k>i$.
It is enough to check this condition between consecutive occurrences of each generator.
The empty word is also canonical and represents the identity element.

It is known~\cite[Theorem~6]{KM} that every element of $K_n$ is represented by exactly one canonical word, which is also its unique representative of minimal length.
Thus the elements of $K_n$ may be identified with canonical words in the generators $\{a_1, \dots, a_n\}$.

The natural homomorphism $K_n\to K_{n+1}$ is injective: mapping $a_{n+1}$ to $1$ and fixing each $a_i$ respects all the defining relations and gives a retraction
\[
\pi:K_{n+1}\longrightarrow K_n.
\]
We henceforth identify $K_n$ with the submonoid of $K_{n+1}$ generated by $a_1,\ldots,a_n$.
A canonical word in~\(K_{n+1}\) may contain \(a_{n+1}\) at most once.
Every element in \(K_{n+1}\setminus K_n\) therefore has a unique canonical expression \(ua_{n+1}v\), with \(u,v\) canonical words in \(a_1,\ldots,a_n\).

We need a slightly stronger observation.
If \(u,v\) are canonical words in \(a_1,\ldots,a_n\), then reducing \(ua_{n+1}v\) to canonical form does not change its suffix $v$.
Indeed, the elementary reductions are
\begin{align}
a_kwa_k&=wa_k
&&\text{if every letter of \(w\) has index greater than \(k\),}
\label{eq:delete-left}\\
a_kwa_k&=a_kw
&&\text{if every letter of \(w\) has index smaller than \(k\).}
\label{eq:delete-right}
\end{align}
A reduction involving occurrences on opposite sides of $a_{n+1}$ can only delete the occurrence on the left: the intervening generator $a_{n+1}$ prevents deletion of the occurrence on the right.
Reductions within the left part likewise leave $v$ unchanged.

\subsection{The suffix order and a parametrization of $K_{n+1} \setminus K_n$}
Equip $K_n$ with its suffix order
\[
x\le y\quad\Longleftrightarrow\quad y\in K_nx.
\]
This is a partial order by $\mathcal L$-triviality, see~\cite[Theorem 22]{KM}.

\begin{proposition}\label{prop:pair-parametrization}
For \(u,u',v\in K_n\),
\begin{equation}\label{eq:deletion-equivalence}
ua_{n+1}v=u'a_{n+1}v
\quad\Longleftrightarrow\quad
uv=u'v.
\end{equation}
Consequently, the assignment
\[
ua_{n+1}v\longmapsto(v,uv)
\]
induces a bijection from \(K_{n+1}\setminus K_n\) onto the set of comparable pairs \(\{(x,y)\in K_n^2:x\le y\}\).
\end{proposition}

\begin{proof}
The forward implication in \eqref{eq:deletion-equivalence} follows by applying the retraction \(\pi\).
Conversely, choose canonical representatives and reduce \(ua_{n+1}v\) and \(u'a_{n+1}v\) to canonical words.
By \eqref{eq:delete-left}
, both retain the suffix \(v\), so they have the form \(wa_{n+1}v\) and \(w'a_{n+1}v\).
Applying \(\pi\) gives
\[
wv=uv=u'v=w'v.
\]
Proposition~29 of~\cite{KM}, after applying the antiautomorphism \(a_i\mapsto a_{n+2-i}\), implies \(w=w'\).
This proves the reverse implication and hence the injectivity of the above assignment.

Its image consists of comparable pairs, since \(v\le uv\).
Conversely, if \(x\le y\), choose \(u\in K_n\) such that~\(y=ux\).
Then \(ua_{n+1}x\) maps to \((x,y)\), proving surjectivity.
\end{proof}

Adding back elements of \(K_n\), Proposition~\ref{prop:pair-parametrization} gives the decomposition
\begin{equation}\label{eq:pair-decomposition}
K_{n+1}\cong
K_n\sqcup\{(x,y)\in K_n^2:x\le y\}.
\end{equation}

The suffix order on \(K_{n+1}\) is particularly easy to describe in these ``coordinates''.

\begin{proposition}\label{prop:transformed-suffix-order}
For \(p,q\in K_n\), one has \(p\le q\) in \(K_{n+1}\) if and only if \(p\le q\) in \(K_n\).
No element of~\(K_{n+1}\setminus K_n\) lies below an element of \(K_n\).
The remaining comparisons are
\begin{align}
p\le(x,y)
&\quad\Longleftrightarrow\quad p\le x,
\label{eq:oldnew}\\
(x,y)\le(u,v)
&\quad\Longleftrightarrow\quad
\bigl(x=u\text{ and }y\le v\bigr)
\ \text{or}\ y\le u.
\label{eq:newnew}
\end{align}
\end{proposition}

\begin{proof}
The first assertion follows from the retraction \(\pi\), while the second holds because multiplication cannot remove a generator from the content of a word.
As for the remaining claims, the defining relations imply $a_{n+1}wa_{n+1}=a_{n+1}w$ for every $w\in K_n$.
Thus multiplication satisfies
\[
m\cdot (x,y)=(x,my),\qquad
(x,y)\cdot m=(xm,ym),\qquad
(r,s)\cdot(x,y)=(ry,sy).
\]
In particular, an element in $K_{n+1} \setminus K_n$ obtained via left multiplication by $p\in K_n$ has first coordinate in~$K_n p$.
Conversely, if $x=bp$ and $y=cx$, then $ca_{n+1}x=(ca_{n+1}b)p$, thus proving \eqref{eq:oldnew}.

We are left with proving \eqref{eq:newnew}.
Multiplying $(x, y)$ on the left by an element in $K_n$ preserves the first coordinate and increases the second.
Multiplying $(x, y)$ on the left by $(r,s)\in K_{n+1}$ yields an element of the form $(ry,sy)$, whose first coordinate is above $y$.
Conversely, if $y\le u\le v$, write $u=ry$ and~$v=tu$.
The comparable pair $(r,tr)$ then satisfies
\[
(r,tr)(x,y)=(u,v).
\]
All claims are then proved.
\end{proof}

\subsection{The Kiselman transform of a finite poset}
\begin{definition}
For any finite poset \(P\), its Kiselman transform is a poset whose underlying set is
\[
 \Kis(P)=P\sqcup\{(x,y)\in P^2:x\le y\},
\]
and its order relation is determined, for \(p,q\in P\) and \((x,y),(u,v)\in\Kis(P)\setminus P\), by
\begin{align*}
p\le q\text{ in }\Kis(P)
&\quad\Longleftrightarrow\quad p\le q\text{ in }P,\\
p\le(x,y)
&\quad\Longleftrightarrow\quad p\le x,\\
(x,y)\le p
&\quad\text{never holds},\\
(x,y)\le(u,v)
&\quad\Longleftrightarrow\quad
\bigl(x=u\text{ and }y\le v\bigr)\ \text{or}\ y\le u.
\end{align*}
\end{definition}

\begin{proposition}\label{prop:K-transform-poset}
The preceding rules define a partial order on \(\Kis(P)\).
In particular, the construction depends only on the poset \(P\), and not on any monoid structure it may carry.
\end{proposition}

\begin{proof}
Reflexivity is immediate.
Notice first that
\[
(x,y)\le(u,v)\quad\Longrightarrow\quad x\le u\ \text{ and }\ y\le v.
\]
Indeed, this is clear when \(x=u\), while in the other case \(x\le y\le u\le v\).
Thus, if two pairs are mutually comparable, their respective coordinates are equal.
Together with antisymmetry in \(P\) and the absence of relations from a pair to an element of \(P\), this proves antisymmetry.

For transitivity, only chains involving pairs are not immediate.
Suppose
\[
(x,y)\le(u,v)\le(r,s).
\]
If \(x=u\), then either \(u=r\), in which case \(x=r\) and \(y\le s\), or \(v\le r\), in which case \(y\le v\le r\).
If instead \(y\le u\), then \(y\le r\): this is immediate when \(u=r\), and otherwise follows from \(y\le u\le v\le r\).
In every case \((x,y)\le(r,s)\).
Finally, if \(p\in P\) and \(p\le(x,y)\le(u,v)\), then \(p\le x\le u\), so \(p\le(u,v)\).
All remaining cases reduce directly to transitivity in \(P\).
\end{proof}

Starting from the one-point poset, the preceding section identifies $\Kis^n\{1\}$ with the suffix poset of $K_n$.
The point of the abstract definition is that all subsequent chain enumeration works for arbitrary $P$.
This is likely to find application in the asymptotic study of cardinalities of finite Hecke-Kiselman monoids, see~\cite{AD, CD}.

\section{Strict-chain vectors and an infinite matrix recursion}\label{chain_vectors}

\subsection{A generating series computation}

Let $\cv_j(P)$ be the number of chains $x_1<\cdots<x_j$ with exactly~$j$ elements.
Define the \emph{chain vector}
\[
 \cv(P)=(\cv_1(P),\cv_2(P),\ldots)^{\mathsf T}
.
\]
For $d=\height(P)$ its support is exactly $\{1,\ldots,d\}$, since every shorter chain occurs as a subset of a longest chain.
In particular $\cv_1(P)=|P|$.
The definition of $\Kis(P)$ immediately gives
\[
 \cv_1(\Kis(P)) = |\Kis(P)|=2\cv_1(P)+\cv_2(P),
\]
since the elements in $\Kis(P)$ of the form $x$ and $(x,x)$ are exactly $2\cv_1(P)$.
We aim to generalize this fact to all other coefficients $\cv_j(\Kis(P))$.

\begin{theorem}\label{thm:chain-vector}
The chain vector satisfies
\[
    \cv(\Kis(P))=\Mat\cv(P),
\]
where \(\Mat\) is the infinite matrix whose entries are
\[
    \Mat_{ij}=\binom{i+j}{2i-j}, \qquad i,j\geq 1,
\]
with the usual convention that a binomial coefficient is zero when its lower index lies outside the interval from \(0\) to its upper index.
The matrix \(\Mat\) is symmetric and admits the factorization
\[
    \Mat=\Bin^{\mathsf T}\Bin,
    \qquad \text{ where }
    \Bin_{ij}=\binom{j}{i-j}.
\]
Moreover,
\[
    \height(\Kis(P))=2\height(P).
\]
\end{theorem}

\begin{proof}
Let \(C\) be a chain in $\Kis(P) = P\sqcup\{(x,y)\in P^2:x\le y\}$.
We shall call the former elements \textbf{singletons} and the latter \textbf{pairs}.
The elements from \(P\) that appear either as singletons in \(C\) or as coordinates of pairs belonging to \(C\) form a chain in \(P\) that we call \emph{support of \(C\)}.

The number of chains in $\Kis(P)$ of length \(i\) having a given support with $j$ elements depends only on \(i\) and \(j\), and not on the choice of the support: denote this number by \(\Mat_{ij}\).
It follows immediately that
\[
    \cv_i(\Kis(P))=\sum_{j\geq 1}\Mat_{ij}\cv_j(P).
\]
It remains to compute the coefficients \(\Mat_{ij}\): we shall do so by describing the corresponding generating series $\sum_{i,j} \Mat_{ij}s^i t^j$.

We thus use \(s\) to record the number of elements of a chain in \(\Kis(P)\), and \(t\) to record the number of distinct elements in its support.
Consider first the portion of a chain consisting of pairs.
A set of consecutive pairs that have the same first coordinate will be (improperly) called a \textbf{group} of the chain.

A group consisting only of a diagonal pair \((x,x)\) contributes one element to the new chain and one element to its support, and therefore has weight
\[
    A=st.
\]
Every other group has the form
\[
    (x,y_1)<\cdots <(x,y_q),
    \qquad x<y_1<\cdots <y_q,
    \qquad q\geq 1,
\]
and may additionally contain the diagonal pair \((x,x)\) at its beginning.
The first coordinate \(x\) contributes~\(t\); each of the \(q\) non-diagonal pairs contributes \(st\); moreover, the optional diagonal pair $(x,x)$ contributes either \(1\) or \(s\).
Hence the generating series for such a group is
\[
    B=t(1+s)\sum_{q\geq 1}(st)^q
      =\frac{st^2(1+s)}{1-st}.
\]

Let \(S\) be the generating series for a nonempty sequence of groups.
Successive groups have strictly increasing first coordinates.
If a group of type \(B\) is followed by another group, its last second coordinate may either be strictly smaller than, or equal to, the first coordinate of the next group.
In the latter case the common support element has been counted twice, so one factor \(t\) must be removed.
This coincidence cannot occur after a diagonal singleton, that is, after a group of type \(A\).
Adding all contributions yields
\[
    S=A+B+\bigl[A+(1+t^{-1})B\bigr]S,
\]
and hence
\[
    S=\frac{A+B}{1-A-(1+t^{-1})B}.
\]

It remains to allow an initial chain of singletons.
Each such element contributes \(st\), so an arbitrary, possibly empty, initial chain of singletons has generating series
\[
    \frac{1}{1-st}.
\]
If pairs follow, the last singleton may coincide with the first coordinate of the first pair.
In that case it contributes a new chain element, hence a factor \(s\), but no new support element.
The two possibilities together therefore contribute
\[
    \frac{1+s}{1-st}.
\]
Including the empty chain, we obtain
\begin{equation}\label{eq:K-bivariate-series}
    1+\sum_{i,j\geq 1}\Mat_{ij}s^it^j
      =\frac{1+(1+s)S}{1-st}
      =\frac{1-st}{1-3st-s^2t-st^2}.
\end{equation}

We now extract the coefficients by an elementary recurrence.
Extend the array by setting \(\Mat_{00}=1\), \(\Mat_{i0}=\Mat_{0i}=0\) for \(i>0\), and \(\Mat_{ij}=0\) whenever one of the indices is negative.
Multiplication of \eqref{eq:K-bivariate-series} by its denominator shows that, for \(i+j\geq 3\),
\begin{equation}\label{eq:K-recurrence}
    \Mat_{ij}
      =3\Mat_{i-1,j-1}+\Mat_{i-2,j-1}+\Mat_{i-1,j-2}.
\end{equation}
The exceptional coefficient of \(st\) gives \(\Mat_{11}-3\Mat_{00}=-1\), and thus \(\Mat_{11}=2\).
Set
\[
    C_{ij}=\binom{i+j}{2i-j}.
\]
We claim that \(C_{ij}\) satisfies the same recurrence.

Set $n=i+j, r=2i-j,$ and use henceforth the notation $[z^n]f(z)$ to denote the coefficient multiplying the monomial $z^n$ in the polynomial or power series $f(z)$.
Since \(C_{ij}=[z^r](1+z)^n\), the right-hand side of \eqref{eq:K-recurrence}, with \(C\) in place of \(\Mat\), is
\begin{align*}
  &3C_{i-1,j-1}+C_{i-2,j-1}+C_{i-1,j-2} \\
  &\quad=[z^r](1+z)^{n-3}
       \bigl(3z(1+z)+z^3+1\bigr) \\
  &\quad=[z^r](1+z)^{n-3}(1+z)^3 \\
  &\quad=[z^r](1+z)^n
       =C_{ij}.
\end{align*}
Here we used only
\[
    1+3z+3z^2+z^3=(1+z)^3.
\]
Moreover, \(C_{00}=1\) and \(C_{11}=2\).
The recurrence determines the coefficients recursively by total degree, and consequently
\[
    \Mat_{ij}=C_{ij}=\binom{i+j}{2i-j}.
\]

The symmetry of \(\Mat\) follows either from this formula and
\[
    \binom{i+j}{2i-j}=\binom{i+j}{2j-i},
\]
or from the factorization below.
Indeed,
\begin{align*}
    (\Bin^{\mathsf T}\Bin)_{ij}
      &=\sum_k\binom{i}{k-i}\binom{j}{k-j} \\
      &=\binom{i+j}{2i-j}
       =\Mat_{ij},
\end{align*}
where the second equality is Vandermonde's identity.

Finally, the binomial formula shows that \(\Mat_{ij}\neq 0\) only if
\[
    2i-j\geq 0
    \qquad\text{and}\qquad
    2i-j\leq i+j,
\]
or equivalently
\[
    \frac{j}{2}\leq i\leq 2j.
\]
The last nonzero entry in column \(j\) is
\[
    \Mat_{2j,j}=\binom{3j}{3j}=1.
\]
Thus every chain of length \(j\) in \(P\) gives rise to a chain of length \(2j\) in \(\Kis(P)\), and no support of size \(j\) can produce a longer chain.
Therefore
\[
    \height(\Kis(P))=2\height(P),
\]
as claimed.
\end{proof}

The first \(9\times 9\) block of \(\Mat\) is
{\small
\[
\Mat=\left(
\begin{matrix}
2&1&0&0&0&0&0&0&0&\hdots\\
1&6&5&1&0&0&0&0&0&\\
0&5&20&21&8&1&0&0&0&\\
0&1&21&70&84&45&11&1&0&\\
0&0&8&84&252&330&220&78&14&\\
0&0&1&45&330&924&1287&1001&455&\\
0&0&0&11&220&1287&3432&5005&4368&\\
0&0&0&1&78&1001&5005&12870&19448&\\
0&0&0&0&14&455&4368&19448&48620&\\
\vdots&&&&&&&&&\ddots
\end{matrix}\right).
\]
}

Similarly, the first \(9\times9\) block of the matrix \(\Bin\) is
{\small
\[
\Bin=\left(
\begin{matrix}
1&0&0&0&0&0&0&0&0&\hdots\\
1&1&0&0&0&0&0&0&0&\\
0&2&1&0&0&0&0&0&0&\\
0&1&3&1&0&0&0&0&0&\\
0&0&3&4&1&0&0&0&0&\\
0&0&1&6&5&1&0&0&0&\\
0&0&0&4&10&6&1&0&0&\\
0&0&0&1&10&15&7&1&0&\\
0&0&0&0&5&20&21&8&1&\\
\vdots&&&&&&&&&\ddots
\end{matrix}\right).
\]
}
Thus each column of $H$ is a shifted row of Pascal's triangle, written vertically.
We have no conceptual explanation of this fact, nor of the symmetry of $M$.

\subsection{Dependence on the height}

\begin{proposition}\label{prop:height-comparison}
Let $P,Q$ be nonempty finite posets of the same height $d$.
Set
\[
 \alpha=\min_{1\leq j\leq d}\frac{\cv_j(P)}{\cv_j(Q)},\qquad
 \beta=\max_{1\leq j\leq d}\frac{\cv_j(P)}{\cv_j(Q)}.
\]
Then $0<\alpha\leq\beta$ and, for every $n\geq0$,
\begin{equation}\label{bounded1}
 \alpha|\Kis^nQ|\leq|\Kis^nP|\leq\beta|\Kis^nQ|.
\end{equation}
In particular,
\begin{equation}\label{bounded2}
 2^{-n/2}\bigl(\lgTwo|\Kis^nP|-\lgTwo|\Kis^nQ|\bigr)\longrightarrow0.
\end{equation}
Along either parity, or any other choice of a subsequence, the normalized logarithms converge for $P$ if and only if they converge for $Q$, and their limits then agree.
\end{proposition}

\begin{proof}
Both the chain vector $\cv(P)$ and $\cv(Q)$ have positive entries precisely in positions $1,\ldots,d$.
This implies that $\alpha\cv(Q)\leq\cv(P)\leq\beta\cv(Q)$ coordinatewise.
Applying the nonnegative matrix $\Mat^n$ yields
\[
\alpha\cv(\Kis^n(Q))\leq\cv(\Kis^n(P))\leq\beta\cv(\Kis^n(Q)),
\]
and taking first coordinates proves \eqref{bounded1}.
Thus
\[
\alpha \leq \frac{|\Kis^n(P)|}{|\Kis^n(Q)|}\leq \beta,
\]
whence \eqref{bounded2} follows immediately.
\end{proof}
The asymptotic problem for every finite poset of height $d$ reduces to a single model.
It can be shown that, as in the Kiselman case, the limits
\begin{equation}
\begin{aligned}
   \ceven^P
   &:=\lim_{n\to\infty}2^{-n}\lgTwo|\Kis^{2n}(P)|,\\
   \codd^P
   &:=\lim_{n\to\infty}2^{-n-\frac12}\lgTwo|\Kis^{2n+1}(P)|,
\end{aligned}
\end{equation}
always exist, but we will not pursue this claim here.

\section{The limits $\ceven$ and $\codd$ are distinct}\label{explicit_bound}

\subsection{From coefficient norms to generating functions}\label{Section:generating}

If \(v=(v_0,v_1,\ldots)^{\mathsf T}\) is a finitely supported vector, write
\[
   \widehat v(s)=\sum_{j\geq0}v_js^j
\]
for its coefficient polynomial.
We keep separate notation for a vector and its polynomial, since both norms of vectors and evaluations of polynomials occur below.

As in the previous section, we extend the array by setting \(\Mat_{00}=\Bin_{00}=1\), and \(\Mat_{ij}=\Bin_{ij} = 0\) when at leat one of the indices $i,j$ is nonpositive but they do not both vanish.
This preserves $\Mat=\Bin^{\mathsf T}\Bin$.
Write $e_j$ for the $j$-th coordinate vector; in particular, $\widehat{e}_1(s) = s$.

Put $h(s)=s+s^2$.
The binomial theorem gives directly
\begin{equation}\label{eq:Bin-action}
 \widehat{\Bin v}(s)
 =\sum_{j\geq0}v_j\sum_{r=0}^j\binom jr s^{j+r}
 =\sum_{j\geq0}v_j h(s)^j
 =\widehat v(h(s)).
\end{equation}
To compute the transpose, put, for $s>0$,
\[
 \delta(s)=\sqrt{s^2+4s},\qquad
 \lambda_\pm(s)=\frac{s\pm\delta(s)}2.
\]
The polynomial $T_j(s)=\sum_{i\geq0}\binom{i}{j-i}s^i$ satisfies $T_0=1$, $T_1=s$ and $T_j=sT_{j-1}+sT_{j-2}$ for $j\geq2$, by Pascal's identity.
Since $\lambda_\pm(s)$ are the roots of $X^2-sX-s=0$, it follows that
\begin{equation}\label{eq:Fibonacci-polynomial}
 \widehat{\Bin^{\mathsf T}e_j}(s)=T_j(s)
 =\frac{\lambda_+(s)^{j+1}-\lambda_-(s)^{j+1}}{\delta(s)}.
\end{equation}
By linearity,
\begin{equation}\label{eq:transpose}
 \widehat{\Bin^{\mathsf T}v}(s)
 =\frac{\lambda_+(s)\widehat v(\lambda_+(s))
       -\lambda_-(s)\widehat v(\lambda_-(s))}{\delta(s)}.
\end{equation}
Define
\[
 \Phi_\pm(s)=h(\lambda_\pm(s))
 =\frac{s^2+3s\pm(s+1)\sqrt{s^2+4s}}2.
\]
Combining the two actions and using $\Mat=\Bin^{\mathsf T}\Bin$ gives
\begin{equation}\label{eq:Mat-action}
 \widehat{\Mat v}(s)
 =\frac{\lambda_+(s)}{\delta(s)}\widehat v(\PhiP(s))
  -\frac{\lambda_-(s)}{\delta(s)}\widehat v(\PhiM(s)).
\end{equation}

Since $K_0$ is the one-point poset, $\cv(K_0)=\unit$.
Set
\[
   u_n:=\Mat^n\unit=\cv(K_n),\qquad
   f_n(s):=\widehat{u_n}(s)
          =\sum_{j\geq1}\cv_j(K_n)s^j.
\]
The first coordinate counts points, so
\begin{equation}\label{eq:A-matrix}
 |K_n|=\unit^{\mathsf T}\Mat^n\unit.
\end{equation}
The height formula in Theorem~\ref{thm:chain-vector} gives \(\height(K_n)=2^n\).
A longest chain contains subchains of every smaller positive cardinality, so
\[
   \supp(u_n)=\{1,\ldots,2^n\}.
\]
Thus \(f_n\) has degree \(2^n\), its coefficients in degrees \(1,\ldots,2^n\) are positive, and \(f_0(s)=s\).

We first express the cardinalities in terms of Euclidean $\ell^2$-norms.
Since \(\Mat\) is symmetric,
\begin{equation}\label{eq:even-norm}
 |K_{2n}|
 =\unit^{\mathsf T}\Mat^{2n}\unit
 =(\Mat^n\unit)^{\mathsf T}(\Mat^n\unit)
 =\norm{u_n}_2^2.
\end{equation}
Similarly, using \(\Mat=\Bin^{\mathsf T}\Bin\),
\begin{equation}\label{eq:odd-norm}
 |K_{2n+1}|
 =\unit^{\mathsf T}\Mat^{2n+1}\unit
 =u_n^{\mathsf T}\Mat u_n
 =u_n^{\mathsf T}\Bin^{\mathsf T}\Bin u_n
 =\norm{\Bin u_n}_2^2.
\end{equation}

If \(v\) has nonnegative entries and support of cardinality at most \(D\), then Cauchy--Schwarz on the left and positivity on the right give the following inequalities between the $\ell^1$ and the $\ell^2$ norms:
\begin{equation}\label{eq:l1-l2-comparison}
   D^{-1}\norm{v}_1^2\leq\norm{v}_2^2\leq\norm{v}_1^2.
\end{equation}

Since the coefficients of \(f_n\) are nonnegative,
\[
   \norm{u_n}_1=\sum_j(u_n)_j=f_n(1).
\]
Applying \eqref{eq:l1-l2-comparison} with \(D=2^n\) to \eqref{eq:even-norm} gives
\begin{equation}\label{eq:even-l1-l2}
   2^{-n}f_n(1)^2\leq |K_{2n}|\leq f_n(1)^2.
\end{equation}

By \eqref{eq:Bin-action}, the coefficient polynomial of \(\Bin u_n\) is
\[
   \widehat{\Bin u_n}(s)=f_n(s+s^2).
\]
It has nonnegative coefficients and degree \(2^{n+1}\); hence
\[
   \norm{\Bin u_n}_1=\widehat{\Bin u_n}(1)=f_n(2).
\]
Applying \eqref{eq:l1-l2-comparison} with \(D=2^{n+1}\) to \eqref{eq:odd-norm} yields
\begin{equation}\label{eq:odd-l1-l2}
   2^{-n-1}f_n(2)^2\leq |K_{2n+1}|\leq f_n(2)^2.
\end{equation}

We now formally define the parity constants mentioned in the introduction:
\begin{equation}\label{eq:parity-constants}
\begin{aligned}
   \ceven
   &:=\lim_{n\to\infty}2^{-n}\lgTwo|K_{2n}|,\\
   \codd
   &:=\lim_{n\to\infty}2^{-n-\frac12}\lgTwo|K_{2n+1}|.
\end{aligned}
\end{equation}
These are precisely the limits of \(2^{-m/2}\lgTwo|K_m|\) along even and odd \(m\), respectively.

For \(s\geq1\), set
\begin{equation}\label{eq:constants-g}
   \esc(s)=\lim_{n\to\infty}2^{-n}\lgTwo f_n(s),
\end{equation}
provided that the limit exists.
Taking logarithms in \eqref{eq:even-l1-l2} and \eqref{eq:odd-l1-l2}, the respective errors are at most \(n\) and~\(n+1\); after the normalizations in \eqref{eq:parity-constants}, both vanish.
Consequently, once the limits in \eqref{eq:constants-g} are shown to exist at \(s=1,2\), we have
\begin{equation}\label{eq:constants-from-g}
   \ceven=2\esc(1),\qquad
   \codd=\sqrt2\,\esc(2).
\end{equation}

\subsection{A scalar iteration and its error bound}
Applying~\eqref{eq:Mat-action} to $u_n$ gives
\begin{equation}\label{eq:operator}
 f_{n+1}(s)=
 \frac{\lambda_+(s)}{\delta(s)}f_n(\PhiP(s))
 - \frac{\lambda_-(s)}{\delta(s)}f_n(\PhiM(s)).
\end{equation}

\begin{proposition}\label{prop:escape}
For $s\ge1$, the limit in \eqref{eq:constants-g} exists and equals
\[
 \esc(s)=\lim_{n\to\infty}2^{-n}\lgTwo\PhiP^{\circ n}(s),
\]
where $\PhiP^{\circ n}$ denotes $n$-fold composition of $\PhiP$.
More precisely,
\begin{equation}\label{eq:tail}
 0\le\esc(s)-2^{-n}\lgTwo \PhiP^{\circ n}(s)
 \le2^{-n}\lgTwo\left(1+\frac4{\PhiP^{\circ n}(s)}\right),
\end{equation}
for every $n$.
\end{proposition}

\begin{proof}
From their explicit expressions,
\[
   \PhiP(s)+\PhiM(s)=s^2+3s,
   \qquad \PhiP(s)\PhiM(s)=-s.
\]
Thus \(\Phi_\pm(s)\) are the roots of the equation
\[
 x^2-(s^2+3s)x-s=0.
\]
Writing \(A=s^2+3s\), the positive root satisfies
\begin{equation}\label{eq:root-identity}
   \PhiP(s)=A+\frac{s}{\PhiP(s)}.
\end{equation}
Since \(\PhiP(s)>A\),
\[
   0<\frac{s}{\PhiP(s)}<\frac{s}{A}=\frac1{s+3}<1.
\]
Thus \(\PhiP(s)>s\), while the product relation gives
\[
   \PhiM(s)=-\frac{s}{\PhiP(s)},\qquad |\PhiM(s)|<1.
\]
Moreover, \eqref{eq:root-identity} gives
\begin{equation}\label{eq:quadratic-bounds}
 s^2\le\PhiP(s)\le s^2+4s.
\end{equation}
Indeed, \(\PhiP>A\geq s^2\), whereas \(\PhiP<A+1\leq s^2+4s\).

Since \(f_n\) has nonnegative coefficients and \(|\PhiM(s)|<\PhiP(s)\),
\[
   |f_n(\PhiM(s))|
   \leq f_n(|\PhiM(s)|)
   \leq f_n(\PhiP(s)).
\]
The weights in \eqref{eq:operator} are positive, have sum one, and their difference satisfies
\[
   \frac{\lambda_+(s)}{\delta(s)}+\frac{\lambda_-(s)}{\delta(s)}=\frac{s}{\delta(s)}
   =\sqrt{\frac{s}{s+4}}\geq\frac1{\sqrt5}
\]
for \(s\geq1\).
The triangle and reverse triangle inequalities in \eqref{eq:operator} therefore give
\[
 \frac1{\sqrt5}f_n(\PhiP(s))\le f_{n+1}(s)\le f_n(\PhiP(s)).
\]
Iterating this inequality \(n\) times and using \(f_0(s)=s\) yields
\begin{equation}\label{eq:comparison}
 5^{-n/2}\PhiP^{\circ n}(s)
 \le f_n(s)\le\PhiP^{\circ n}(s).
\end{equation}

Let $z_j(s)=\PhiP^{\circ j}(s)$ and $b_j(s)=2^{-j}\lgTwo z_j(s)$.
By \eqref{eq:quadratic-bounds},
\[
   z_j(s)^2\leq z_{j+1}(s)
   \leq z_j(s)^2\left(1+\frac4{z_j(s)}\right).
\]
Taking base-\(2\) logarithms and multiplying by \(2^{-j-1}\) gives
\[
 0\le b_{j+1}(s)-b_j(s)
 \le 2^{-j-1}\lgTwo\left(1+\frac4{z_j(s)}\right).
\]
Since \(z_j\) is increasing, summing from \(j=n\) onward and using \(\sum_{j\geq n}2^{-j-1}=2^{-n}\) proves convergence and \eqref{eq:tail}.
Finally, by \eqref{eq:comparison} the two logarithms differ by at most \(\frac n2\lgTwo 5=o(2^n)\).
Therefore
\[
 \lim_{n\to\infty}2^{-n}\lgTwo f_n(s)
 =
 \lim_{n\to\infty}2^{-n}\lgTwo z_n(s),
\]
which proves the claimed expression for \(\esc(s)\).
\end{proof}

\subsection{Separation of the constants}

\begin{theorem}
The constants defined in \eqref{eq:parity-constants} satisfy $\codd<\ceven$.
More precisely,
\[
 \codd< 2.5062576 < 2.5062593 <\ceven.
\]
\end{theorem}

\begin{proof}
For $s=1,2$, set
\[
 z_0(s)=s,\qquad z_{j+1}(s)=\PhiP(z_j(s)),\qquad
 B_4(s)=\frac1{16}\lgTwo z_4(s).
\]
Four iterations will suffice.
By \eqref{eq:tail} and $\ln(1+x)\leq x$,
\[
 0\leq \esc(s)-B_4(s)
 \leq \frac1{16}\lgTwo\left(1+\frac4{z_4(s)}\right)
 \leq \frac1{4z_4(s)\ln2}.
\]
In particular, $2B_4(1)$ is a lower bound for $\ceven$.
To separate the constants, we therefore only need to bound the error in the odd approximation.

Let us evaluate the four-step approximations.
Iterating
\[
 \Phi(x)=\frac{x^2+3x+(x+1)\sqrt{x^2+4x}}2
\]
with sufficient precision, we obtain
\[
 1085610<z_4(1)<1085611,\qquad
 343341201<z_4(2)<343341202.
\]
The former estimate implies
\[
 \ceven\geq 2B_4(1)
 >\frac18\lgTwo1085610
 >2.5062593,
\]
whereas the lower bound from the latter gives
\[
0\le \codd-\sqrt2\,B_4(2)
\le \frac{\sqrt2}{16}
\log_2\!\left(1+\frac4{z_4(2)}\right)
<\frac{\sqrt2}{4\cdot343341201\ln2}
<1.49\cdot10^{-9}
\]
yielding
\[
 \codd
 <\sqrt2\,B_4(2)+1.49\cdot10^{-9}
 <\frac{\sqrt2}{16}\lgTwo343341202+1.49\cdot10^{-9}
 <2.5062576.
\]
\end{proof}

The numerical precision required here is modest.
To see why errors remain controlled, set~$t=\lambda_+(x)$, so that
\[
 x=\frac{t^2}{t+1},\qquad \PhiP(x)=t(t+1).
\]
It follows that
\[
 \frac{d\ln\PhiP(x)}{d\ln x}
 =\frac{2t+1}{t+2}<2,
\]
and hence
\[
 |\ln\PhiP(x)-\ln\PhiP(y)|
 \leq2|\ln x-\ln y|
 \qquad(x,y>0).
\]
Thus a relative error is amplified by at most a factor of two at each step, while the final logarithm is divided by $16$.
Iterating the approximation six times instead of four yields the more precise estimate
\begin{align*}
 2B_6(1)&=2.506259580995114968521706627901\ldots,\\
 \sqrt2 B_6(2)&=2.506257586488454455996719227420\ldots.
\end{align*}
The rational enclosures for these approximations have width less than $2\cdot10^{-49}$.
Adding the scaled tails gives
\begin{align*}
 2.50625958099511496852170662
 &<\ceven<
 2.50625958099511496852170676,\\
 2.50625758648845445599671922
 &<\,\codd\,\!<
 2.50625758648845445599671924,
\end{align*}
hence
\[
 \ceven-\codd=0.00000199450666051252\ldots.
\]
In the files attached to this arXiv submission, we also supply a Python script certifying the above bounds.

\section{Computation of $|K_n| \bmod 4$}

Recall that Andren\v{s}ek proved in~\cite{An26Zero} that $n$ and $|K_n|$ have opposite parity.
In this section we compute $|K_n|\bmod4$ by showing that $f_m(s)\bmod2$, for $m>0$, is obtained by adding pairs of consecutive powers of $s$ to the monomial $s$ (respectively $s^2$) when $m$ is even (resp.\ odd).

The computation of $|K_n|$ then follows from the fact that $|K_{2m}|$ and $|K_{2m+1}|$ coincide with the squared~$\ell^2$-norms of $u_m$ and $\Bin u_m$, respectively, and that $x^2\bmod4$ is determined by $x\bmod2$.

\begin{lemma}\label{lemma:chain-vector-mod-two}
For every $n\geq0$, the chain polynomial of $K_n$ satisfies
\[
 f_n(s)\equiv
 1+\sum_{\substack{0\leq i\leq n\\ \binom ni\text{ odd}}}
 \bigl(s^{2^i-1}+s^{2^i}\bigr)
 \pmod2.
\]
\end{lemma}

\begin{proof}
Work over $\mathbb F_2 = \mathbb Z/2\mathbb Z$, identifying coefficient vectors with their polynomials.
Recall that in \eqref{eq:Mat-action} we considered the linear map
\[
{\MM}: \widehat{v}(s) \mapsto \widehat{\Mat v}(s)
\]
that we extended to all polynomials by imposing ${\MM}(1) = 1$.
If now
\[
 A_i(s)=s^{2^i}+s^{2^i-1}\qquad \text{ for } i \geq 0 ,
\]
we claim that
\begin{equation}\label{eq:mod-two-Pascal}
 \MM A_i=A_i+A_{i+1}.
\end{equation}

To prove this, set $Q_j(s)=\MM(s^j)$.
The binomial formula for the entries of $\Mat$ gives, modulo $2$,
\[
 \sum_{j\geq0}Q_j(s)t^j
 =\frac{1+st}{1+(s^2+s)t+st^2}.
\]
Consequently,
\[
 Q_0=1,\qquad Q_1=s^2,\qquad
 Q_j=(s^2+s)Q_{j-1}+sQ_{j-2}\quad(j\geq2).
\]
For our convenience, let us set $Q_{-1} = 1$ so that the Fibonacci-type recurrence also holds for $j = 1$.
Let~$\alpha,\beta$ be the roots of
\[
 X^2+(s^2+s)X+s=0
\]
in a suitable algebraic extension of $\mathbb F_2(s)$.
For $j\geq1$ we have
\[
 Q_j+Q_{j-1}
 =\frac{s+1}{s}\bigl(\alpha^j+\beta^j\bigr).
\]
Indeed, both sides satisfy the same second-order recurrence and agree for $j=0,1$.

If $j=2^i$, the Frobenius identity yields
\[
 \alpha^j+\beta^j=(\alpha+\beta)^j=(s^2+s)^j.
\]
Hence
\[
 \begin{aligned}
 \MM A_i
 &=Q_j+Q_{j-1}\\
 &=s^{j-1}(s+1)^{j+1}\\
 &=s^{j-1}(s+1)(s^j+1)\\
 &=(s^j+s^{j-1})+(s^{2j}+s^{2j-1})
 =A_i+A_{i+1},
 \end{aligned}
\]
proving \eqref{eq:mod-two-Pascal}.

Since $f_0(s)=s=1+A_0(s)$ and $f_{n+1}=\MM f_n$, induction using \eqref{eq:mod-two-Pascal} and Pascal's identity proves
\[
 f_n(s)=1+\sum_{i=0}^n\binom ni A_i(s)
 \qquad\text{in }\mathbb F_2[s].
\]
Only the odd binomial coefficients contribute, as claimed.
\end{proof}

\begin{theorem}\label{thm:mod-four}
For every $n\geq1$,
\[
\begin{aligned}
 |K_n|&\equiv 2     \mod4 &&\qquad \text{if $n$ is odd},\\
 |K_n|&\equiv n-1   \mod4 &&\qquad \text{if $n$ is even}.
\end{aligned}
\]
\end{theorem}

\begin{proof}
For $g\in\mathbb F_2[s]$, let $w(g)$ denote its number of nonzero coefficients.
Since an integer square is congruent modulo $4$ to its parity, the norm formulas give
\[
 |K_{2m}|\equiv w(f_m),\qquad
 |K_{2m+1}|\equiv w\bigl(f_m(s+s^2)\bigr)
 \pmod4,
\]
where the polynomials on the right are reduced modulo $2$.

Put $\varepsilon_i=\binom mi\bmod2$, with $\varepsilon_i=0$ for $i>m$.
By Lemma~\ref{lemma:chain-vector-mod-two},
\[
 f_m(s)=s+\sum_{i\geq1}\varepsilon_i
 \bigl(s^{2^i}+s^{2^i-1}\bigr).
\]
The displayed pairs have disjoint supports, and only the pair with $i=1$ contains $s$.
Thus
\[
 w(f_m)=1+2\sum_{i\geq1}\varepsilon_i-2\varepsilon_1.
\]
If $m\geq1$, then
\[
 \sum_{i\geq0}\varepsilon_i
 \equiv\sum_{i=0}^m\binom mi=2^m\equiv0\pmod2.
\]
Since $\varepsilon_0=1$ and $\varepsilon_1\equiv m\pmod2$, it follows that
\[
 |K_{2m}|\equiv3-2m\equiv 2m-1\pmod4.
\]

As for odd values of $n$, put $h(s)=s+s^2$ and $A_i(s)=s^{2^i}+s^{2^i-1}$.
If $i\geq1$ and $r=2^i$, then
\[
 A_i(h(s))
 =s^r+s^{2r}+s^{r-1}(1+s)^{r-1}
 =s^{r-1}+\sum_{j=r+1}^{2r-2}s^j+s^{2r}.
\]
Here we used $(1+s)^{r-1}=1+s+\cdots+s^{r-1}$ over $\mathbb F_2$; the displayed sum is empty when $r=2$.
These polynomials have pairwise disjoint supports and $w(A_i(h))=2^i$.
Only $A_1(h)=s+s^4$ meets the support of $h$, at the monomial $s$.
Consequently, for every $m\geq0$,
\[
 \begin{aligned}
 w(f_m(h))
 &=2+\sum_{i\geq1}\varepsilon_i2^i-2\varepsilon_1\\
 &=2+\sum_{i\geq2}\varepsilon_i2^i
 \equiv2\pmod4.
 \end{aligned}
\]
\end{proof}

\appendix
\section{Computation of $|K_n|$}\label{app:recurrences}

Recall that $u_n=\cv(K_n)=\Mat^n\unit$.
Equations~\eqref{eq:even-norm} and~\eqref{eq:odd-norm} give
\[
 |K_{2m}|=\norm{u_m}_2^2,\qquad
 |K_{2m+1}|=\norm{\Bin u_m}_2^2.
\]
Notice that the computation of $|K_{2m}|$ and $|K_{2m-1}|$ only requires that of $\Mat^m \unit$ and $\Bin \Mat^{m-1} \unit$, which are supported on their first $2^m$ entries.
For instance, $|K_{30}|$ requires only the $32768$ coefficients of $\Mat^{15}\unit$ and does not require enumerating $K_{30}$.

If $Q_0 = 1$ and $Q_j(s) = \widehat{\Mat e_j}(s), j \geq 1$, equation~\eqref{eq:K-bivariate-series} immediately gives $Q_1(s) = 2s + s^2$ and
\begin{equation}\label{eq:column-recurrence}
\begin{gathered}
Q_j(s)=(3s+s^2)Q_{j-1}(s)+sQ_{j-2}(s)
\qquad \text{ whenever } j\geq2.
\end{gathered}
\end{equation}
These \(Q_j\) are precisely the columns of \(\Mat\), encoded as polynomials in \(s\).

Let $v=(v_0,\ldots,v_d)^{\mathsf T}$ be a finitely supported vector.
Then $\widehat{\Mat v}(s)=\sum_{j=0}^d v_jQ_j(s)$, and since the $Q_j$ satisfy the three-term recurrence~\eqref{eq:column-recurrence}, this sum can be evaluated by Clenshaw's algorithm~\cite{Clenshaw}: setting~$b_{d+1}=b_{d+2}=0$ and $b_j=v_j+(3s+s^2)b_{j+1}+s\,b_{j+2}$ for $j=d,\ldots,1$, one has
\begin{equation}\label{eq:backward-M-evaluation}
  \widehat{\Mat v}(s)=v_0+(2s+s^2)\,b_1(s)+s\,b_2(s),
\end{equation}
where the last step uses $Q_1=2s+s^2$ and $Q_2-(3s+s^2)Q_1=s$.
Similarly, by~\eqref{eq:Bin-action}, $\widehat{\Bin v}(s)=\widehat v(s+s^2)$ is computed by Horner's scheme.
The resulting procedure is summarized in Algorithm~\ref{alg:cardinalities}.

Both procedures use only additions, shifts of coefficient arrays (multiplication by \(s\) or \(s^2\)), and multiplication by the small integer \(3\), hence provide an efficient exact implementation over the integers.
They also admit matrix-free implementations using only \(O(d)\) working memory.
The backward recurrence for \(\Mat\) avoids constructing or storing the \(O(d^2)\) relevant binomial entries of the matrix, although its arithmetic complexity with ordinary coefficient arrays remains \(O(d^2)\), as does that of a direct streamed computation from the binomial formula.

Implementations of Algorithm~\ref{alg:cardinalities} in Python and in C, together with the exact values of $|K_n|$ for $n$ up to 30, are provided in the files attached to this arXiv submission; their output agrees with the previously tabulated values through $n=14$.
The complete run through $n=30$ of the Python implementation took~$28$ minutes on a standard laptop; the C implementation, which is faster but uses more memory, finished the same task in $43$ seconds on the same hardware.

\begin{algorithm}[H]
\caption{Cardinalities of $K_n$ for $0 \le n \le N$}
\label{alg:cardinalities}
\begin{algorithmic}[1]
\Statex All arithmetic takes place in $\mathbb{Z}[s]$; for $p \in \mathbb{Z}[s]$,
  $\|p\|^2$ denotes the sum of the squares of its coefficients.
\Statex
\Function{Step}{$v$}
  \Comment{$\sum_j v_j Q_j(s)$, by Clenshaw's recurrence}
  \State write $v = \sum_{j=0}^{d} v_j s^j$ with $d = \deg v$
  \State $b_1 \gets 0$; \quad $b_2 \gets 0$
  \For{$j = d$ \textbf{down to} $1$}
    \State $(b_1, b_2) \gets \bigl(v_j + (3s+s^2)\,b_1 + s\,b_2,\; b_1\bigr)$
  \EndFor
  \State \Return $v_0 + (2s+s^2)\,b_1 + s\,b_2$
\EndFunction
\Statex
\Function{Smooth}{$v$}
  \Comment{$\sum_j v_j (s+s^2)^j$, by Horner's rule}
  \State write $v = \sum_{j=0}^{d} v_j s^j$ with $d = \deg v$
  \State $b \gets 0$
  \For{$j = d$ \textbf{down to} $0$}
    \State $b \gets v_j + (s+s^2)\,b$
  \EndFor
  \State \Return $b$
\EndFunction
\Statex
\Procedure{Cardinalities}{$N$}
  \State $v \gets s$
  \For{$n = 0$ \textbf{to} $N$}
    \If{$n$ is even}
      \State \textbf{output} $\|v\|^2$
    \Else
      \State \textbf{output} $\|\Call{Smooth}{v}\|^2$
      \State $v \gets \Call{Step}{v}$
    \EndIf
  \EndFor
\EndProcedure
\end{algorithmic}
\end{algorithm}

\end{document}